\documentclass[reqno,12pt,reqno]{amsart}

\usepackage{amsfonts, amsthm, amsmath, amssymb, enumerate, verbatim, bm, cite, manfnt, multicol}
\usepackage{hyperref}
\usepackage[utf8]{inputenc}
\usepackage{subfiles}
\hypersetup{colorlinks=false}

\RequirePackage{mathrsfs} \let\mathcal\mathscr

\numberwithin{equation}{section}

\newtheorem{theorem}{Theorem}[section]
\newtheorem{lemma}[theorem]{Lemma}

\theoremstyle{definition}
\newtheorem*{ack}{Acknowledgements}
\newtheorem{remark}[theorem]{Remark}

\renewcommand{\rho}{\varrho}

\newcommand{\ZZ}{\mathbf{Z}}

\newcommand{\QQ}{\mathbf{Q}}
\newcommand{\RR}{\mathbf{R}}

\renewcommand{\leq}{\leqslant}

\renewcommand{\geq}{\geqslant}

\newcommand{\bs}{\boldsymbol}

\newcommand{\x}{\mathbf{x}}
\newcommand{\y}{\mathbf{y}}
\renewcommand{\v}{\mathbf{v}}
\newcommand{\uu}{\mathbf{u}}
\newcommand{\cc}{\mathbf{c}}
\newcommand{\dd}{\mathbf{d}}
\newcommand{\z}{\mathbf{z}}
\newcommand{\w}{\mathbf{w}}
\renewcommand{\b}{\mathbf{b}}

\newcommand{\ve}{\varepsilon}

\DeclareMathOperator{\supp}{supp}

\DeclareMathOperator{\sgn}{sgn}

\DeclareMathOperator{\mat}{Mat}
\DeclareMathOperator{\diag}{Diag}

\DeclareMathOperator{\sln}{SL_n}

\newcommand{\sumstar}{\sideset{}{^*}\sum}
\newcommand{\sumflat}{\sideset{}{^\flat}\sum}
\newcommand{\jacobi}[2]{\left(\frac{#1}{#2}\right)}

\renewcommand{\mod}[1]{\hspace{-2.9mm}\bmod{#1}}

\newcommand{\hf}[1]{h\left(r,#1\right)}

\begin{document}

\title{Counting solutions to quadratic polynomials}

\author{V.\ Vinay Kumaraswamy}

\address{Department of Mathematics, KTH, 10044 Stockholm \\}
\email{vinay.visw@gmail.com} 


\begin{abstract}
Using the circle method, we obtain asymptotic formulae for the number of integer solutions to certain quadratic polynomials that 
are uniform in the coefficients of the polynomial.
\end{abstract}

\date{\today}

\maketitle

\section{Introduction}

Let $F(x_1,\ldots,x_n) \in \ZZ[x_1,\ldots,x_n]$ be a quadratic polynomial with integer coefficients. We write 
\begin{equation}\label{eq:fqln}
F(\x) = Q(\x) + L(\x) + N,
\end{equation} 
where $Q(\x)$ is an indefinite quadratic form of rank $r$ with $5 \leq r \leq n$, $L(x_1,\ldots,x_n)$ is a linear form and $N$ is an integer. In this article we will give an asymptotic formula for the counting function
$$
N(P,w) = \sum_{\substack{\x \in \ZZ^n \\ F(\x) = 0}}w(P^{-1}\x)
$$  
as $P \to \infty$, where $w(\x)$ is a smooth function with compact support. Importantly, this result is uniform in the size of the coefficients of $F$, and it 
is a key ingredient in~\cite{KR24} to count rational points on cubic hypersurfaces. Our result is a generalisation of work of Browning and Dietmann~\cite[Proposition 1]{BD08} to 
the setting of general quadratic equations.

\subsection{Statements of the main results}
Let $M = \left(M_{i,j}\right)\in \mat_{n\times n}(\QQ)$ be a symmetric matrix such that $Q(x_1,\ldots,x_n) = \x^tM\x$. Since $M$ is symmetric, there exists an orthogonal 
transformation $R \in O_n(\RR)$ such that $R^{t} M R = \diag(\lambda_1,\ldots,\lambda_{r},0,\ldots,0)$, where $\lambda_i$ are the non-
zero eigenvalues of $M$. Set $\|F\| = \max_{i,j}|M_{i,j}|.$ Since the matrix entries 
of $M$ are bounded by $\|F\|$ we have the estimate $|\lambda_i| \leq n \|F\|$.

Let $c > 0$ be a small real number. Let $w_0$ be a non-negative smooth function with compact support in $[-2c,2c]$ such that 
$w_0(x) = 1$ for $x \in [-c,c]$. 

Let $\bs{\xi} = (\xi_1,\ldots,\xi_n) \in \RR^n$ be a non-singular solution to the equation $Q(\x) = 0$. Define
\begin{equation}\label{eq:stanwtfn}
w_2(\x) = \prod_{i=1}^n w_0(x_i-\xi_i).
\end{equation}

If $\lambda_1,\ldots,\lambda_{r}$ are the non-zero eigenvalues of $Q(\x)$, 
let $\tau = (\tau_1,\ldots,\tau_{r}) \in \RR^{r}$ be a non-zero real solution to 
$$
Q_{\sgn}(\x) = \sum_{i=1}^{r}\frac{\lambda_i}{|\lambda_i|}x_i^2 = 0.
$$ 
Observe that such a solution exists because $Q_{\sgn}(\x)$ is indefinite.
Set
\begin{equation*}
w_1(\x) = \prod_{i=1}^{n-h}w_0(x_i-\tau_i).
\end{equation*}
Also, let
\begin{equation}\label{eq:wtbrd}
w_Q(\x) = w_3(R^t\x),
\end{equation}
where
\begin{equation*}
w_3(\x) = w_1(|\lambda_1|^{1/2}x_1,\ldots,|\lambda_{r}|^{1/2}x_{r},\|F\|^{1/2}x_{r+1},\ldots,\|F\|^{1/2}x_n).
\end{equation*}
Let
$$
\sigma_{\infty}(Q_{\sgn},w_1) \int_{\RR}\int_{\RR^n}w_1(\x)e(-\theta Q_{\sgn}(\x))\, d\x \,d\theta
$$
and
$$
\sigma_{\infty}(Q,w_2) = \int_{\RR}\int_{\RR^n}w_2(\x)e(-\theta Q(\x))\, d\x \,d\theta.
$$
denote the singular integrals associated with $Q_{\sgn}$ and $Q$. Let 
\begin{equation}\label{eq:singseriesf}
\mathfrak{S}(F) = \prod_p\lim_{t \to \infty}p^{-t(n-1)}\#\left\{\x \bmod{p^t} : F(\x) \equiv 0 \bmod{p^t}\right\}
\end{equation}
denote the singular series.

Define 
\begin{equation*}
\kappa = \kappa_n = \begin{cases}0 & \mbox{ if $2 \mid r,$} \\ 1 &\mbox{ otherwise.}
\end{cases}
\end{equation*}

\begin{theorem}\label{propbh19}
Suppose that $F(\x)$ is as in~\eqref{eq:fqln} where $Q(\x)$ is an indefinite quadratic form of rank $r \geq 5$. Let $\lambda_1,\ldots,\lambda_{r}$ 
be the non-zero eigenvalues of $Q$ and let $\Delta = \prod_{i=1}^r |\lambda_i|$. 
Assume that the equation $F(\x) \equiv 0 \bmod{p^t}$ is soluble for all prime powers $p^t$. 
Assume for any $A \geq 1$ we have 
$$
\sup_{|\x| \leq A} |L(\x)| \ll \|F\|^2,
$$ where the implied constant depends only on $A$. 
Let $w_Q(\x)$ be as in~\eqref{eq:wtbrd}. Let $\eta > 0$. Suppose that $P \geq 1$ is such that 
$$
\|F\| \ll \min\left\{P^{2/3-\eta},|\lambda_i|^{1/4}P^{1/2-\eta}\right\}\text{ and } |N| \ll P^{2-\eta}.
$$ 
Define 
$$
\mathcal{E}(P) = \Delta^{-\frac{1}{2r}}\|F\|^{\frac{r}{2}}P^{n-\frac{r}{2}-\frac{1}{2}} + 
\Delta^{-\frac{1}{r}+\frac{\kappa}{2r}}\|F\|^{\frac{r}{2}-1}P^{n-\frac{r}{2}-\frac{\kappa}{2}}.
$$
Then for any $\ve > 0$ we have 
$$N(P,w_Q) 
= 
\frac{\sigma_{\infty}(Q_{\sgn},w_1)\mathfrak{S}(F)P^{n-2}}{\|F\|^{\frac{n-r}{2}}\Delta^{\frac{1}{2}}}+ O_{\ve,\eta}\left(\|F\|^{\ve}P^{\ve}\mathcal{E}(P)\right).$$
Moreover, we have that $\sigma_{\infty}(Q_{\sgn},w_1) \gg 1$ and $\mathfrak{S}(F) \gg_{\ve} \Delta^{-\frac{r}{r-4}-\ve}$.
\end{theorem}
\begin{remark}
The estimate $\mathfrak{S}(F) \gg_{\ve} \Delta^{-\frac{r}{r-4}-\ve}$ follows from a straightforward modification of~\cite[Proposition 2]{BD08}.
\end{remark}

Note that Theorem~\ref{propbh19} implies, in particular, that a solution to the equation $F = 0$ exists, provided that 
$Q$ is indefinite of rank at least $5$ and that the equation $F \equiv 0\bmod{m}$ is soluble for every
integer $m$, thereby establishing the Hasse principle for such quadratic polynomials. This recovers a classical result of Siegel~\cite{S51} (see also work
of Watson~\cite{W61} for
a more elementary proof).

We also draw the reader's attention to important work of Dietmann~\cite{D03} that established a very general form of Theorem~\ref{propbh19}. Whereas
Dietmann requires only that $r \geq 3$ and that $F$ is everywhere locally soluble, our result
is stronger for those quadratic polynomials $F$ for which the product of the non-zero
eigenvalues of $Q$ is `large' in absolute value. This is precisely what has been ustilised in~\cite{KR24}. 

We prove Theorem~\ref{propbh19} by using the smooth $\delta$-function form of the circle method due to Heath-Brown~\cite[Theorem 3]{HB96} that
is based on earlier work of Duke, Friedlander and Iwaniec~\cite{DFI93}. The weight function $w_Q$ is modelled on the analogous weight function
constructed in~\cite{BD08}, which not only yields a good lower bound for the main term in $N(P, w_Q)$, but it also minimises the size of the error term. 

We end this introduction by remarking that by counting solutions
to the equation $F=0$ weighted by $w_2$, we obtain the following generalisation of~\cite[Theorem 2.1]{BH19}, whose proof 
we omit on account of its similarity to the proof of Theorem~\ref{propbh19}.
\begin{theorem}\label{propbh199}
Suppose that $r \geq 5$ and let $w$ be as in~\eqref{eq:stanwtfn}. Assume that $|N| \ll \|F\|^3$ and that for any $A \geq 1$ we have 
$$
\sup_{|\x| \leq A} |L(\x)| \ll \|F\|^{2},
$$ where the implied constant depends only on $A$. Suppose that $\|F\| \leq P$. Define 
$$\mathcal{E}(P) = \|F\|^{\frac{n}{2}+\frac{r}{4}-1-\frac{\kappa}{4}}P^{n-\frac{r}{2}-\frac{\kappa}{2}} + \|F\|^{\frac{n}{2}+\frac{r}{4}-\frac{1}{4}}P^{n-\frac{r}{2}-\frac{1}{2}}.$$
Then for any $\ve > 0$ we have $$N(P,w_2) = \sigma_{\infty}(F,w)\mathfrak{S}(F)P^{n-2} + O\left(\|F\|^{\ve}P^{\ve}\mathcal{E}(P)\right).$$
\end{theorem}
It is interesting to note that while Theorem~\ref{propbh19} is useful in obtaining lower bounds, one can obtain good upper bounds for the number of 
integer solutions to quadratic polynomials using Theorem~\ref{propbh199}, provided that $\|F\|$ is not too large in terms of $P$. Removing this 
constraint would have striking applications to the Dimension Growth Conjecture for cubic hypersurfaces, which remains unsolved. 

\begin{ack}This work was commenced while I was a Visiting Fellow at the School of Mathematics in TIFR Mumbai, where I was supported by
a DST Swarnajayanti fellowship. I am currently supported by the grant KAW 2021.0282 from the Knut and Alice Wallenberg Foundation. 
Part of this work was also supported by the Swedish Research Council under grant
no. 2021-06594 while I participated in the programme in Analytic Number Theory at the Institute Mittag-Leffler in 2024.
\end{ack}

\section{Counting solutions using the smooth $\delta$ function}
Using~\cite[Theorem 3]{HB96} with $Q^2 = P^2$ to detect the equation $F(\x) = 0$, we get
\begin{equation}\label{eq:npw}
N(P,w) = \frac{1}{P^2}\sum_{\cc \in \ZZ^n}\sum_{q=1}^{\infty}q^{-n}S_q(\cc)I_q(\cc) + O(1),
\end{equation}
where 
\begin{equation*}
\begin{split}
S_q(\cc) &= \sumstar_{a \bmod{q}}\sum_{\b \bmod{q}}e_q(aF(\b) + \b.\cc) \\
I_q(\cc) &= \int_{\RR^n} w_Q(P^{-1}\x) h\left(\frac{q}{P},\frac{F(\x)}{P^2}\right)e_q(-\cc.\x)\, d\x.
\end{split}
\end{equation*}
\subsection{Estimates for $I_q(\cc)$}
Let $r = q/P$ and $\v = P\mathbf{c}/q$. We begin by writing
$$I_q(\cc) = P^nI_r^{*}(\v),$$ where 
\begin{align*}
I_r^{*}(\v) &= \int_{\RR^n}w_Q(\x)\hf{Q(\x)+\frac{L(\x)}{P}+\frac{N}{P^2}}e(-\x.\v)\,d\x
\end{align*}
and $w_Q(\x)$ is given in~\eqref{eq:wtbrd}.  Let $G(\x) = Q(\x) + L(\x)/P + N/P^2$. Observe that for $\x$ lying in the support of $w_Q$, we have 
$$
L(\x) \ll \|F\|^{3/2}+\|F\|^2\sum_{i=1}^{r}|\lambda_i|^{-1/2}.
$$ 
Since $\|F\|^{3/2}+\|F\|^2\sum_{i=1}^{r}|\lambda_i|^{-1/2} \ll P$ by the hypothesis of Theorem~\ref{propbh19}, we have that 
$
|G(\x)| \leq c_{w,F}
$
for some constant $c_{w,F} > 0$, whenever $\x \in \supp w_Q$. Let $U: \RR \to \RR$ be a smooth function supported on 
$[-1-c_{w,F},1+c_{w,F}]$ such that $U(x) = 1$ for $x \in [-c_{w,F},c_{w,F}]$. Then 
\begin{align}\label{eq:irvu}
I_r^*(\mathbf{v}) = \int_{\RR^n}w_Q(\x)\hf{G(\x)}U(G(\x))e(-\x.\v)\,d\x.
\end{align}

\begin{lemma}\label{lemmahard}
Let $\lambda_1,\ldots,\lambda_r$ be the non-zero eigenvalues of $Q$. Suppose that 
$$
\sup_{|\x| \leq A} |L(\x)| \ll_A \|F\|^2,
$$
for any $A \geq 1$. If $\cc \neq \mathbf{0}$ then we have the following estimates for $I_q(\mathbf{c})$. We have for any $N \geq 0$ that
\begin{equation}\label{lemmaparts}
I_q(\cc) \ll \Delta^{-1/2}\|F\|^{-(n-r)/2}\frac{P^{n+1}}{q}\frac{\|F\|^{\frac{N}{2}}}{|\cc|^N}.
\end{equation}
In addition, we have 
\begin{equation}\label{eq:hard}
I_q(\cc) \ll \Delta^{-1/2}\|F\|^{-n/2+3r/4-1/2}P^{n-\frac{r}{2}+1}q^{\frac{r}{2}-1}|\cc|^{-\frac{r}{2}+1}.
\end{equation}
\end{lemma}
\begin{proof}

Let $D \in \mat_{n\times n}$ denote the diagonal matrix 
$$
D = \diag (|\lambda_1|^{-1/2},\ldots,|\lambda_r|^{-1/2},\|F\|^{-1/2},\ldots,\|F\|^{-1/2}).
$$ 
Making the change of variables $R^t\x = \z$ in~\eqref{eq:irvu}, we get that
\begin{equation*}
\begin{split}
I_r^*(\v) &= \int_{\RR^n}w_3(\z)\hf{G(R\z)}e(-\z.R^t\v)\,d\z. \\
\end{split}
\end{equation*}
Recalling the definition of $D$, we get that
\begin{equation*}
\begin{split}
I_r^*(\v)&= \|F\|^{-(n-r)/2}\prod_{i=1}^{r}|\lambda_i|^{-1/2}\times \\
&\quad \quad \int_{\RR^n}w_1(\z)\hf{G(DR\z)}e(-\z.R^tD\v)\,d\z.
\end{split}
\end{equation*}

If $\cc \neq \mathbf{0}$, then it follows from the estimate $|\lambda_i| \ll \|F\|$ that $|R^tD\v| \gg \|F\|^{-1/2}|\v|$. Integration by parts and~\cite[Lemma 5]{HB96} yield the estimate
\begin{align*}
I_r^{*}(\v) &\ll_A \Delta^{-1/2}\|F\|^{-(n-r)/2}|\v|^{-A}r^{-1-A} \\
&\ll \Delta^{-1/2}\|F\|^{-(n-r)/2}r^{-1}\left(|\mathbf{c}|/\sqrt{\|F\|}\right)^{-A},
\end{align*}
for any $A > 0$. This completes the proof of~\eqref{lemmaparts}

Let $p(t)$ denote the Fourier transform of the compactly supported smooth function $g(x) = h(r,x)U(x)$. Then we have
\begin{align*}
I_r^{*}(\v) &= \|F\|^{-(n-r)/2}\Delta^{-1/2}\times \\
&\quad\int_{\RR} p(t)\int_{\RR^n} w_1(\x) e(tG(DR\x) - \x.R^tD\v)\, d\x \, dt \\
&= \|F\|^{-(n-r)/2}\Delta^{-1/2}\int_{\RR}p(t)J(t,\v)\, dt,
\end{align*}
where we have set
\begin{equation*}
J(t,\v) = \int_{\RR^n} w_1(\x) e(tG(DR\x) - \x.R^tD\v)\, d\x.
\end{equation*}
Appealing to~\cite[Lemma 10]{HB96}, we see that 
\begin{equation}\label{eq:smallt}
J(t,\v) \ll_M \|F\|^{\frac{M}{2}}|\v|^{-M}
\end{equation} 
  if $|t| \ll |\v|/\|F\|^{1/2}.$
In the remaining range we proceed as follows. Applying Parseval's theorem as in~\cite[Lemma 3.2]{HBP17} we get that 
$$
J(t,\v) \ll \left(\int_{\RR^n} |\widehat{w_1}(\x)|\, d\x\right)\sup_{\y \in \RR^n}\int_{[-1,1]^n}e(tG(DR\x) - \x.\y)\, d\x.
$$ 
It is easy to see that 
$\int_{\RR^n} |\widehat{w_1}(\x)|\, d\x \ll 1$. 
Observe that $Q(DR\x) = Q_{\sgn}(\x)$. As a result, 
\begin{align*}
J(t,\v) &\ll \prod_{i=1}^{r}\sup_{y \in \RR}\int_{-1}^1e(\frac{\lambda_i}{|\lambda_i|}tx^2+xy)\, dx\, \times \\
&\quad \prod_{i={r}+1}^n\sup_{y \in \RR}\int_{-1}^1e(xy)\, dx.
\end{align*}
Therefore, we obtain the bound
\begin{equation}\label{eq:jtv}
J(t,\v) \ll \min\left\{1,|t|^{-\frac{r}{2}}\right\}.
\end{equation}
By the preceding estimate and~\eqref{eq:smallt} we get for any integer $N > 0$ that
\begin{align*}
I_r^*(\v) &\ll \|F\|^{-(n-r)/2}\Delta^{-1/2}\left\{\|F\|^{\frac{N}{2}}\int_{|t| \ll |\v|/\|F\|^{1/2}} |\v|^{-N} \, dt\right. \\
&\left. \quad \quad +\int_{|t| \gg |\v|/\|F\|^{1/2}}|t|^{-\frac{r}{2}} \, dt\right\}.
\end{align*}
The estimate~\eqref{eq:hard} now follows by choosing $N$ as follows. Set $N = r/2$ if $r$ is even. If $r$ is odd, choose $N = \frac{r-1}{2}$ if $|\v| \ll \sqrt{Y}$ and $N = \frac{r+1}{2}$ otherwise. 
\end{proof}

\begin{lemma}\label{iq0}
Suppose that
$$
\sup_{|\x|\leq A} |L(\x)| \ll_A \|F\|^{3/2}+\|F\|^2\sum_{i=1}^{r}|\lambda_i|^{-1/2}. 
$$
for any $A \geq 1$. Let $0 < N < r/2-1$ be an integer. Suppose also that 
$$
\|F\|^{3/2}+\|F\|^2\sum_{i=1}^{r}|\lambda_i|^{-1/2} \leq P^{1-\eta} \text{ and } \quad |N| \leq P^{2-\eta}
$$for some $\eta > 0$. Then 
$$
I_q(\mathbf{0}) = \|F\|^{-(n-r)/2}\Delta^{-\frac{1}{2}}P^n\left(\sigma_{\infty}(Q_{\sgn},w_1) + O(P^{-\eta/2}) + O(q^NP^{-N})\right).
$$
Moreover,
$1 \ll \sigma_{\infty}(Q_{\sgn},w_1) \ll 1$.
\end{lemma}
\begin{proof}
We begin by observing that the singular integral satisfies the bounds $1 \ll \sigma_{\infty}(Q_{\sgn},w_1) \ll 1$ 
by construction of the weight function $w_1$. We will now show that the expression for $I_q(\mathbf{0})$ holds.

We have
$
I_q(\mathbf{0}) = P^n I_r^*(\mathbf{0}),
$
with
\begin{align*}
I_r^*(\mathbf{0}) &= \int_{\RR^n}w_Q(\x)\hf{Q(\x)+\frac{L(\x)}{P}+\frac{N}{P^2}}\,d\x \\
&= \|F\|^{-(n-r)/2}\prod_{i=1}^{r}|\lambda_i|^{-1/2}\times \\
&\quad \int_{\RR^n}w_1(\x)\hf{Q_{\sgn}(\x) + \frac{L(DR\x)}{P} + \frac{N}{P^2}} \, d\x,
\end{align*}
where we have used the definition of $w_Q$ from~\eqref{eq:wtbrd} in the last line.

For $\theta \in \RR$ let
$$
L(\theta) = \int_{\RR}U(t)h(r,t)e(\theta t)\,dt.
$$
Then by~\cite[Equation 2.10]{BH19}, we have for any $N > 0$ that
\begin{equation*}
L(\theta) = 1 + O_N((q/P)^N) + O_N((q/P)^N|\theta|^N).
\end{equation*}
Also let
\begin{align*}
J(\theta,\mathbf{0}) &= \int_{\RR^n}w_1(\x)e\left(-\theta \left(Q_{\sgn}(\x) + \frac{L(DR\x)}{P} + \frac{N}{P^2}\right)\right) \, d\x \,.
\end{align*}
Appealing to~\cite[Equation 2.9]{BH19}, we deduce that
\begin{align*}
I_r^*(\mathbf{0}) &= \|F\|^{-(n-r)/2}\prod_{i=1}^{r}|\lambda_i|^{-1/2}\lim_{\delta \downarrow 0}\int_{\RR}\left(\frac{\sin (\pi \delta \theta)}{\pi \delta \theta}\right)^2J(\theta,\mathbf{0})L(\theta)\, d\theta. 
\end{align*}
As a result, for any $N <r/2-1$ we have that
\begin{align*}
&= \|F\|^{-(n-r)/2}\Delta^{-1/2}\lim_{\delta \downarrow 0}\int_{\RR}\left(\frac{\sin (\pi \delta \theta)}{\pi \delta \theta}\right)^2J(\theta,\mathbf{0}) \, d\theta \,+ \\
&\quad 
O\left((q/P)^N\|F\|^{-(n-r)/2}\Delta^{-1/2}\lim_{\delta \downarrow 0}\int_{\RR}(1+|\theta|^N)\min\left\{1,|\theta|^{-r/2}\right\}\, d\theta \right) \\
&=  \|F\|^{-(n-r)/2}\Delta^{-1/2}\lim_{\delta \downarrow 0}\int_{\RR}\left(\frac{\sin (\pi \delta \theta)}{\pi \delta \theta}\right)^2J(\theta,\mathbf{0}) \, d\theta \,+ \\
&\quad \quad O\left(\|F\|^{-(n-r)/2}\Delta^{-1/2}q^N P^{-N}\right).
\end{align*}
Letting $\delta \to 0$ we get that
\begin{equation}\label{eq:24411}
\begin{split}
I_r^*(\mathbf{0}) &= \|F\|^{-(n-r)/2}\Delta^{-1/2}\int_{\RR}J(\theta,\mathbf{0})\, d\theta \, \\
&\quad+ O\left(\|F\|^{-(n-r)/2}\Delta^{-1/2}q^N P^{-N}\right).
\end{split}
\end{equation}
By~\eqref{eq:jtv}, we have that
$$
\int_{\RR} J(\theta,\mathbf{0}) \, d\theta = \int_{|\theta|\leq P^{\eta}}J(\theta,0) \, d\theta + O(P^{-\eta(r/2-1)}).
$$
As a result, 
\begin{align*}
\int_{\RR} J(\theta,\mathbf{0}) \, d\theta &=\int_{\RR^n}w_1(\x) \times\\
&\quad\int_{|\theta|\leq P^{\eta}}e\left(-\theta \left(Q_{\sgn}(\x) + \frac{L(DR\x)}{P} + \frac{N}{P^2}\right)\right) \, d\theta \, d\x \\
& \quad\quad +O(p^{-\eta(r/2-1)}).
\end{align*}
By the hypothesis of the lemma, we have that 
$$
L(DR\x)/P + N/P^2 \ll P^{-\eta},
$$ 
for 
any $\x \in \supp w_1$. As a result, 
$$
e(-\theta (Q_{\sgn}(\x) + L(DR\x)/P + N/P^2)) = e(-\theta Q_{\sgn}(\x))(1 + O(P^{-\eta})),
$$
whence 
$$
\int_{|\theta| \leq P^{\eta}} J(\theta,\mathbf{0})\, d\theta = \int_{|\theta| \leq P^{\eta}/2}\int_{\RR^n}w_1(\x)e(-\theta Q_{\sgn}(\x))\, d\theta \, d\x + O(P^{-\eta/2}).
$$
Notice that the argument used to deduce~\eqref{eq:jtv} also shows that
$$
\int_{\RR^n}w_1(\x)e(-\theta Q_{\sgn}(\x))\, \ll \min\left\{1,|\theta|^{-r/2}\right\}.
$$
Therefore, we have
$$
\int_{\RR} J(\theta,\mathbf{0}) \, d\theta = \int_{\RR}\int_{\RR^n}w_1(\x)e(-\theta Q_{\sgn}(\x))\, d\theta \, d\x + O(P^{-\eta/2}).
$$
As $I_q(\mathbf{0}) = P^n I_r^*(\mathbf{0})$, by~\eqref{eq:24411} we may conclude for any $0 < N < r-1/2$ that
$$
I_q(\mathbf{0}) = \|F\|^{-(n-r)/2}\Delta^{-\frac{1}{2}}P^n\left(\sigma_{\infty}(Q_{\sgn},w_1) + P^{-\eta/2} + q^NP^{-N}\right).
$$
This completes the proof of the lemma.
\end{proof}

\subsection{Analysis of the exponential sum}
In this section, we will analyse the exponential sum $S_q(\cc)$. We begin by recording some notation that will be used repeatedly in this
section. 

Let $r$ denote the rank of the quadratic form $Q(\x) = \x^tM\x$. 
Completing squares, we find that there exists a matrix $S \in \sln(\QQ)$ such that $S^tMS = \diag(\beta_1,\ldots,\beta_{r},0,\ldots,0),$ where $\beta_i$ are non-zero integers. Let 
\begin{equation}\label{eq:dualquadform}
Q^*(\x) = \x^tS^{-1}\diag(\beta_1^{-1},\ldots,\beta_{r}^{-1},0,\ldots,0)(S^t)^{-1}\x
\end{equation} 
denote the quadratic form dual to $Q(\x)$.
Put
$\mathcal{D} = \prod_{i=1}^{r}\beta_i$. Then $\mathcal{D} \ll \|F\|^{r}$. 

Clearing denominators, we may suppose that $S$ has integer entries (with a potentially different set of $\beta_i)$. Then if $p  \nmid 2\mathcal{D}$, we have for any $k \geq 1$ that $S^tMS \equiv \diag(\beta_1,\ldots,\beta_{r},0,\ldots,0) \bmod{p^k}$, with $(\beta_i,p)=1$. Also recall $Q^*(\x)$ from~\eqref{eq:dualquadform}. 
We begin by recording the following multiplicativity result.
\begin{lemma}\label{multlemma}
Let $q = uv$ with $(u,v) = 1$. Then $$S_q(\cc) = S_u(\overline{v}\cc)S_v(\overline{u}\cc),$$ where $\overline{u}$ (resp. $\overline{v}$) is the multiplicative inverse of $u$ (resp. $v$) modulo $v$ (resp. $u$). 
\end{lemma}
Next we record the following estimate for $S_q(\cc)$.
\begin{lemma}\label{lemmacs}
For any $q \geq 1$, we have $$S_q(\cc) \leq (q^{r},\mathcal{D})^{\frac{1}{2}}q^{1+n-\frac{r}{2}}.$$
\end{lemma}
\begin{proof}
By Cauchy-Schwarz we get
$$|S_q(\cc)|^2 \leq \phi(q) \sumstar_{a\bmod{q}}\sum_{\uu,\v \bmod{q}}e_q(a\left\{F(\uu)-F(\v)\right\}+(\uu-\v).\cc).$$
Make the change of variables $\uu-\v = \w$. Then we get
$$|S_q(\cc)|^2 \leq \phi(q) \sumstar_{a\bmod{q}}\sum_{\w \bmod{q}}e_q(aQ(\w)+\w.\cc)\sum_{\v \bmod{q}}e_q(a\nabla Q(\w).\v).$$
There exist $A, B \in \sln(\ZZ)$ so that $M$ can be put in Smith canonical form, i.e. $A^tMB = \diag(\alpha_1,\ldots,\alpha_{r},0,\ldots,0)$, for some non-zero integers $\alpha_i$. Replacing $\v$ by $A\v$ and $\w$ by $B\w$ we get
$$|S_q(\cc)|^2 \leq \phi(q)q^n \sumstar_{a\mod{q}}\sumflat_{\w \bmod{q}}e_q(a\w^tB^tMB\w+\w.B^t\cc),$$
where $\sumflat$ denotes restriction to $\w$ such that $w_i \equiv 0 \bmod{q/(q,\alpha_i)}$ for $1 \leq i \leq r$. The result follows.
\end{proof}

\begin{lemma}\label{lemmagaussrankdrop}
Suppose that $p \nmid 2\mathcal{D}$. Let $\bs{l}$ be a vector such that $L(\x) = \bs{l}.\x$, where $L(\x)$ is as in~\eqref{eq:fqln}. Set $\dd = S^t \mathbf{c}$ 
and $\mathbf{m} = S^t \bs{l}$. Then
\begin{enumerate}

\item $S_{p^k}(\cc)$ vanishes unless $v_p(d_i) = v_p(m_i)$ for each $r < i \leq n$.

\item If $(m_{r+1},\ldots,m_n)=\mathbf{0}$, then $S_{p^k}(\cc)$ vanishes unless $d_j \equiv 0 \bmod{p^k}$ for $r+1 \leq j \leq n$, in which case, we get
$$|S_{p^k}(\cc)| = \begin{cases}p^{k(n-\frac{r}{2})}|S(4N-Q^*(\bs{l}), Q^*(\cc);p^k)| &\mbox{ if $2 \mid kr$} \\ 
p^{k(n-\frac{r}{2})}|T(4N-Q^*(\bs{l}),Q^*(\cc);p^k)| &\mbox{ if $2 \nmid kr$},\end{cases}$$ where
$$S(m,n;q) = \sumstar_{x \bmod{q}}e_q(mx+n\overline{x})$$ to be a Kloosterman sum and let $$T(m,n;q) = \sumstar_{x \bmod{q}}\jacobi{x}{q}e_q(mx+n\overline{x})$$ be a 
Sali\'{e} sum. 
\item Let $J \subset \{r+1,\ldots,n\}$ such that $m_j \neq 0$ whenever $j \in J$. Assume that $v_p(m_j) = 0$ for $j \in J$. Then $$|S_{p^k}(\cc)| \leq p^{k(n-\frac{r}{2})}.$$
Furthermore, there exist congruence classes $a(j,p^k)$ modulo $p^k$ for $j \in J$ such that $S_{p^k}(\mathbf{c})$ vanishes unless $d_j \equiv a(j,p^k) \bmod{p^k}$ for each $r+1 \leq j \leq n $. 
\end{enumerate}
\comment{Then for any $\ve > 0$ we have the bound
$$|S_{p^k}(\cc)| \ll p^{k(n-\frac{r}{2}+\mathcal{D})+k\ve}$$ 
where
\begin{align*}
\mathcal{D} = \begin{cases} 1 &\mbox{ $N = 0$ and $\widehat{F}(\cc) \equiv 0 \bmod{p}$} \\ 
0 &\mbox{ $N = 0$ and $\widehat{F}(\cc) \not\equiv 0 \bmod{p}$} \\
0 &\mbox{ $N \neq 0$ and $\widehat{F}(\cc) \equiv 0 \bmod{p}$} \\
\frac{1}{2} &\mbox{ $N \neq 0$ and $\widehat{F}(\cc) \not\equiv 0 \bmod{p}$},\end{cases} 
\end{align*}
if $r$ is even. If $r$ is odd, then $S_p(\cc)$ vanishes if $N = 0$ and $\widehat{F}(\cc) \equiv 0 \bmod{p}$. If this not the case, then $\mathcal{D} = \frac{1}{2}$.}

\end{lemma}
\begin{proof}
Recall that
\begin{equation*}
S_{p^k}(\cc) = \sumstar_{a \bmod{p^k}}e_{p^k}(aN)\sum_{\b \bmod{p^k}}e_{p^k}(a\left\{Q(\b)+L(\b)\right\}+\b.\cc).
\end{equation*}
Since $p \nmid 2\mathcal{D}$, we see that 
$S^tMS \equiv \diag(\beta_1,\ldots,\beta_{r},0,\ldots,0)$ modulo $p^k$, with $(p,\beta_i)=1$. Note that $\beta_i$ is unique up to multiplication by squares.
Replacing $\b$ by $S\b$ in $S_q(\cc)$ we get
\begin{align*}
S_{p^k}(\cc) &= \sumstar_{a \bmod{p^k}}e_{p^k}(aN)\sum_{\b \bmod{p^k}}e_{p^k}(a\left\{Q(S\b)+L(S\b)\right\}+\b.S^t\cc) \\
&= \sumstar_{a \bmod{p^k}}e_{p^k}(aN)\prod_{i=1}^{r}\sum_{b_i \bmod{p^k}}e_{p^k}(a(\beta_i b_i^2 + m_ib_i) + b_i d_i) \times \\
&\quad \quad \prod_{i=r+1}^n \sum_{b_i \bmod{p^k}} e_{p^k}(a m_ib_i + b_id_i) \\
&= \sumstar_{a \bmod{p^k}}e_{p^k}(aN) S_1(a) S_2(a),
\end{align*}
say. 
We have
\begin{align*}
S_1(a) = \jacobi{\mathcal{D}}{p^k}p^{\frac{kr}{2}}\jacobi{a}{p}^{kr}e_{p^k}\left(-\overline{4a}\sum_{i=1}^r\overline{\beta_i}(d_i+am_i)^2\right).
\end{align*}
Turning to $S_2(a)$, we see that it vanishes unless $am_i+d_i \equiv 0 \bmod{p^k}$ for $r+1\leq i \leq n$, and if this is the case, we get $S_2(a) = p^{n-k}$ and the first statement of the lemma follows. If $(m_{r+1},\ldots,m_n) = \bs{0}$, then
\begin{align*}
|S_{p^k}(\cc)| &= p^{k(n-\frac{r}{2})}\bigg\vert\sumstar_{a\bmod{p^k}}\jacobi{a}{p}^{kr}e_{p^k}(aN - \overline{4a}\sum_{i=1}^{r}\overline{\beta_i}(d_i+am_i)^2)\bigg\vert
\end{align*}
and the second statement of the lemma follows upon observing that 
\begin{align*}
(\dd+\mathbf{m})^t\diag(\beta_1^{-1},\ldots,\beta_r^{-1},0,\ldots,0)(\dd + \mathbf{m}) &=Q^{*}(\cc+a\bs{l}).
\end{align*}

If $(m_{r+1},\ldots,m_n) \neq \mathbf{0}$ then the condition $am_i + d_i \equiv 0 \bmod{p^k}$ implies that
$v_p(m_i) = v_p(d_i)$ and that $a \equiv - d_i/m_i \bmod{p^k/(p^k,m_i)}$ whenever $m_i \neq 0$. The final statement of the lemma follows and this completes the proof.
\end{proof}

\begin{lemma}\label{lemmaavg1}
Let $\dd = S^t\cc$ and $\mathbf{m} = S^t\bs{l}.$ Suppose that $m_i = d_i = 0$ for every $r+1\leq i \leq n$. Suppose that $4N - Q^*(\bs{l}) = 0$. 
\begin{enumerate}
\item If $Q^{*}(\cc) = 0$, we have
\begin{align*}
\sum_{q \leq x} S_q(\cc) \ll_{\ve} \mathcal{D}^{\frac{1}{2}-\frac{1}{r} + \frac{\kappa}{2r} + \ve} x^{n-\frac{r}{2} + 2 - \frac{\kappa}{2}+\ve}.
\end{align*}
\item If $Q^{*}(\cc) \neq 0$, we have
\begin{align*}
\sum_{q \leq x} S_q(\cc) \ll_{\ve} \mathcal{D}^{\frac{1}{2}- \frac{\kappa}{2r} + \ve} x^{n-\frac{r}{2} + 1 + \frac{\kappa}{2}+\ve}.
\end{align*}
\end{enumerate}
\end{lemma}
\begin{proof}
Using Lemmas~\ref{multlemma} and~\ref{lemmacs} we get
$$
\sum_{q \leq x}S_q(\cc) \leq \sum_{\substack{q_2 \leq x \\ q_2 \mid (2\mathcal{D})^{\infty}}}q_2^{1+n-\frac{r}{2}}(q_2^{r},\mathcal{D})^{1/2}
\sum_{\substack{q_1 \leq x/q_2 \\ (q_1,2\mathcal{D})=1}}|S_{q_1}(\cc)|.
$$
By Lemma~\ref{lemmagaussrankdrop}, the inner sum is
\begin{align*}
\sum_{\substack{q_1 \leq x/q_2 \\ (q_1,2\mathcal{D})=1}}|S_{q_1}(\cc)| &\ll 
\begin{cases} 
\sum_{\substack{q_1 \leq x/q_2 \\ (q_1,2\mathcal{D})=1}}q_1^{1+n-\frac{r}{2}} &\mbox{ if $r$ is even} \\
\sum_{\substack{q_1 \leq x/q_2 \\ q_1 = \square \\ (q_1,2\mathcal{D})=1}}q_1^{1+n-\frac{r}{2}} &\mbox{ if $r$ is odd}.
\end{cases} \\
&\ll \left(\frac{x}{q_2}\right)^{2+n-\frac{r}{2}-\frac{\kappa}{2}+\ve}.
\end{align*}
The first part of the lemma follows by summing over $q_1.$

To prove the second statement, we begin by letting $D = 2\mathcal{D} Q^{*}(\cc).$ Then we can write
\begin{align*}
\sum_{q \leq x}S_q(\cc) \leq \sum_{\substack{q_2 \leq x \\ q_2 \mid D^{\infty}}}q_2^{1+n-\frac{r}{2}}(q_2^{r},\mathcal{D})^{1/2}
\sum_{\substack{q_1 \leq x/q_2 \\ (q_1,D)=1}}|S_{q_1}(\cc)|.
\end{align*}
By Lemma~\ref{lemmagaussrankdrop}, we have
$|S_{q_1}(\cc)| = q_1^{n-\frac{r}{2}}$ if $r$ is even and $|S_{q_1}(\cc)| = q_1^{n-\frac{r}{2}+\frac{1}{2}}$ if $r$ is odd. Therefore,
\begin{align*}
\sum_{q \leq x} S_q(\cc) &\ll \sum_{\substack{q_2 \leq x \\ q_2 \mid D^{\infty}}}q_2^{1+n-\frac{r}{2}}(q_2^{r},\mathcal{D})^{1/2}\left(\frac{x}{q_2}\right)^{n-\frac{r}{2}+\frac{\kappa}{2}+1} \\
&\ll x^{n-\frac{r}{2}+\frac{\kappa}{2}+1}\sum_{\substack{q_2 \leq x \\ q_2 \mid D^{\infty}}}q_2^{-\frac{\kappa}{2}}(q_2^{r},\mathcal{D})^{1/2}.
\end{align*}
The second statement of the lemma follows by summing over $q_2$.
\end{proof}

\begin{lemma}\label{lemmaavg2}
Let $\dd = S^t\cc$ and $\mathbf{m} = S^t\bs{l}.$ Suppose that $m_i = d_i = 0$ for every $r+1\leq i \leq n$. Suppose that $4N - Q^*(\bs{l}) \neq 0$. 

\begin{enumerate}
\item If $Q^{*}(\cc) = 0$, we have
\begin{align*}
\sum_{q \leq x} S_q(\cc) \ll_{\ve} \mathcal{D}^{\frac{1}{2}- \frac{\kappa}{2r} + \ve} x^{n-\frac{r}{2} + 1 + \frac{\kappa}{2}+\ve}.
\end{align*}
\item If $Q^{*}(\cc) \neq 0$, we have
\begin{align*}
\sum_{q \leq x} S_q(\cc) \ll_{\ve} \mathcal{D}^{\frac{1}{2}-\frac{1}{2r}+\ve}x^{n-\frac{r}{2}+\frac{3}{2}+\ve}.
\end{align*}
\end{enumerate}

\end{lemma}

\begin{proof}
The proof of the first statement is similar to the proof of the second statement of Lemma~\ref{lemmaavg1}. The only difference is that we must take $D = 2\mathcal{D}(4N-Q^*(\bs{l}))$.

For the second, let $D = 2\mathcal{D}(4N-Q^*(\bs{l}))Q^*(\cc)$. Then by using Lemmas~\ref{multlemma} and~\ref{lemmacs}, we get
$$
\sum_{q \leq x} S_q(\cc) \ll \sum_{\substack{q_2 \leq x \\ q_2 \mid D^{\infty}}}q_2^{1+n-\frac{r}{2}}(q_2^{r},\mathcal{D})^{1/2}
\sum_{\substack{q_1 \leq x/q_2 \\ (q_1,D)=1}}|S_{q_1}(\cc)|.
$$
By the Weil bound for exponential sums, we get $|S_{q_1}(\cc)| \ll q_1^{n-\frac{r}{2}+\frac{1}{2}+\ve}$. As a result, we have
\begin{align*}
\sum_{q \leq x} S_q(\cc) &\ll x^{n-\frac{r}{2}+\frac{3}{2}+\ve}\sum_{\substack{q_2 \leq x \\ q_2 \mid D^{\infty}}}q_2^{-\frac{1}{2}}(q_2^{r},\mathcal{D})^{1/2} \\
&\ll \mathcal{D}^{\frac{1}{2}-\frac{1}{2r}+\ve}x^{n-\frac{r}{2}+\frac{3}{2}+\ve}.
\end{align*}
This completes the proof of the lemma.
\end{proof}

\begin{lemma}\label{lemmaavg3}
Let $\dd = S^t\cc$ and $\mathbf{m} = S^t\bs{l}.$ Suppose that $m_i = 0$ for every $r+1 \leq i \leq n$. If $d_i \neq 0$ for some $r+1 \leq i \leq n$. If $|\cc| \leq \|F\|^{A}$, for some $A > 0$, then we have
\begin{align*}
\sum_{q \leq x}S_q(\cc) \ll_{\ve} \mathcal{D}^{\frac{1}{2}+\ve}x^{1+n-\frac{r}{2}+\ve}.
\end{align*}
\end{lemma}
\begin{proof}
Invoking Lemmas~\ref{multlemma} and~\ref{lemmacs} we get
\begin{align*}
\sum_{q \leq x}S_q(\cc) \leq \sum_{\substack{q_2 \leq x \\ q_2 \mid (2\mathcal{D})^{\infty}}}q_2^{1+n-\frac{r}{2}}(q_2^{r},\mathcal{D})^{\frac{1}{2}}\sum_{\substack{q_1 \leq x/q_2 \\ (q_1,2\mathcal{D}) = 1}}|S_{q_1}(\cc)|.
\end{align*}
Suppose that $d_j \neq 0$ where $r+1 \leq j \leq n$. Then by Lemma~\ref{lemmagaussrankdrop} we get that $S_{q_1}(\cc)$ vanishes unless $q_1 \mid d_j$. Therefore we find that
$$
\sum_{\substack{q_1 \leq x/q_2 \\ (q_1,2\mathcal{D}) = 1}}|S_{q_1}(\cc)| =  \sum_{\substack{q_1 \leq x/q_2 \\ (q_1,2\mathcal{D}) = 1 \\ q_1 \mid d_j}}|S_{q_1}(\cc)| \ll (|d_j|)^{\ve}\left(\frac{x}{q_2}\right)^{n-\frac{r}{2}+1}.
$$
The lemma follows by summing over $q_2$.
\end{proof}

\begin{lemma}\label{lemmaavg4}
Let $\dd = S^t\cc$ and $\mathbf{m} = S^t\bs{l}.$ Suppose that $m_i \neq 0$ for some $r+1 \leq i \leq n$. Then we have
\begin{align*}
\sum_{q \leq x}S_q(\cc) \ll \mathcal{D}^{\frac{1}{2}+\ve}x^{1+n-\frac{r}{2}+\ve}.
\end{align*}
\end{lemma}
\begin{proof}
Suppose that $m_j \neq 0$ for some $r+1 \leq j \leq n$. Then we once again have
\begin{align*}
\sum_{q \leq x}S_q(\cc) \leq \sum_{\substack{q_2 \leq x \\ q_2 \mid (2\mathcal{D} m_j)^{\infty}}}q_2^{1+n-\frac{r}{2}}(q_2^{r},\mathcal{D})^{\frac{1}{2}}\sum_{\substack{q_1 \leq x/q_2 \\ (q_1,2\mathcal{D} m_j) = 1}}|S_{q_1}(\cc)|.
\end{align*}
By Lemma~\ref{lemmagaussrankdrop}, the inner sum is bounded by $(x/q_2)^{1+n-\frac{r}{2}}.$ The lemma follows from summing over $q_2$.
\end{proof}

\subsection{Proof of Theorem~\ref{propbh19}}

Recall that $S \in \sln(\QQ)$ such that $S^tMS = \diag (\beta_1,\ldots,\beta_{r},0,\ldots,0)$. Let $\mathcal{D} = \prod_{i=1}^{r}\beta_i$. 
Let $\Delta$ denote the absolute value of the product
of the non-zero eigenvalues $\lambda_i$ of $M$. Then $\mathcal{D} \ll \Delta$. 
Let $\mathbf{m} = S^t\bs{l}.$ Let $\ve > 0$ and $C = \|F\|^{\frac{1}{2}}P^{\ve}.$ 
By~\eqref{eq:npw} and Lemma~\ref{lemmaparts} we have 
\begin{equation*}
N(P,w) = M(P) + E(P) + E'(P) + O(1),
\end{equation*} where 
$$M(P) = \frac{1}{P^2}\sum_{q=1}^{\infty}q^{-n}S_q(\mathbf{0})I_q(\mathbf{0}),\, \, E(P) = \frac{1}{P^2}\sum_{\substack{\cc \in \ZZ^n \\ 1 \leq |\cc| \ll C}}\sum_{q=1}^{\infty}q^{-n}S_q(\mathbf{c})I_q(\mathbf{c})$$
and
$$
E'(P) = \frac{1}{P^2}\sum_{\substack{\cc \in \ZZ^n \\ |\cc| \gg C}}\sum_{q=1}^{\infty}q^{-n}S_q(\mathbf{c})I_q(\mathbf{c}).
$$
We begin by estimating $E'(P)$. Let $N > 0$ be any integer. We have by~\eqref{lemmaparts} and the trivial estimate $|S_q(\cc)|\leq q^n$.
\begin{align*}
E'(P) &\ll \|F\|^{-\frac{n}{2}}P^{n-1}\sum_{q \ll P}\frac{1}{q}
\sum_{\substack{\cc \in \ZZ^n \\ |\cc| \gg C}}\frac{\|F\|^{\frac{N+1}{2}}}{|\cc|^{N}} \\
&\ll C^{n-N}\|F\|^{\frac{1}{2}-\frac{n-N}{2}}P^{n-1+\ve} \ll 1
\end{align*}
by taking $N$ large enough. 

We will now estimate $M(P)$ and $E(P)$. There are two natural cases to consider based on whether the tuple $(m_{r+1},\ldots,m_n) = \mathbf{0}$ or not. We begin by discussing the former case.
\subsubsection{Case 1: $(m_{r+1},\ldots,m_n) = \mathbf{0}$} We begin by analysing the main term $M(P)$. Let $J < r-1/2$ be a positive integer. 
By Lemma~\ref{iq0} we get that
\begin{equation*}
\begin{split}
M(P) &= \sigma_{\infty}(Q_{\sgn},w_1)P^{n-2}\|F\|^{-(n-r)/2}\Delta^{-1/2}\sum_{q \ll P}\frac{S_q(\mathbf{0})}{q^n} \\
&\quad  + O\left(\|F\|^{-(n-r)/2}\Delta^{-1/2}P^{n-2-\eta/2}\sum_{q \ll P}\frac{|S_q(\mathbf{0})|}{q^{n}}\right) \\
&\quad \quad + O\left(\|F\|^{-(n-r)/2}\Delta^{-1/2}P^{n-2-J}\sum_{q \ll P}\frac{|S_q(\mathbf{0})|}{q^{n-J}}\right) \\
&= M_1 + E_1+E_2,
\end{split}
\end{equation*}
say. Observe that the estimates in Lemmas~\ref{lemmaavg1} and~\ref{lemmaavg2} ensure that the sum $\sum_{q=1}^{\infty}q^{-n}S_q(\mathbf{0})$ 
converges absolutely if $r \geq 5$. Moreover, we have the estimate 
\begin{equation}\label{eq:absconvsingseries}
\sum_{q \leq x}\frac{S_q(\mathbf{0})}{q^n} = \sum_{q = 1}^{\infty}\frac{S_q(\mathbf{0})}{q^n} + O_{\ve}(\mathcal{D}^{1/2}x^{2-r/2+\ve}).
\end{equation}
Recall the constant $\mathfrak{S}(F)$ from~\eqref{eq:singseriesf}. It 
is well-known that 
$$
\sum_{q = 1}^{\infty}\frac{S_q(\mathbf{0})}{q^n} = \mathfrak{S}(F).
$$
Therefore, we have 
\begin{align*}
M_1 = \|F\|^{-(n-r)/2}\Delta^{-1/2}\mathfrak{S}(F)\sigma_{\infty}(Q_{\sgn},w_1)P^{n-2} + O(P^{n-r/2+\ve}).
\end{align*}
Next, using~\eqref{eq:absconvsingseries} once again, we get that $E_1 \ll P^{n-r/2-\eta/2}$. Taking $J = \frac{r+\kappa}{2}-2$, we have by
Lemmas~\ref{lemmaavg1} and~\ref{lemmaavg2} that
\begin{equation}\label{eq:mpe0}
\begin{split}
\sum_{q \ll P}\frac{|S_q(\mathbf{0})|}{q^{n-J}} &\ll \sum_{\substack{x \ll P \\ x \text{ dyadic}}}x^{-n+\frac{r}{2}+\frac{\kappa}{2}-2}\sum_{q \ll x}|S_q(\mathbf{0})| \\
&\ll \|F\|^{\ve}P^{\ve}\mathcal{D}^{\frac{1}{2}-\frac{\kappa}{2r}}.
\end{split}
\end{equation}
As a result, we get $E_2 \ll P^{n-r/2+\ve}$.
Thus we have shown that
\begin{equation}\label{eq:mp}
\begin{split}
M(P) &= \|F\|^{-(n-r)/2}\Delta^{-1/2}\mathfrak{S}(F)\sigma_{\infty}(Q_{\sgn},w_1)P^{n-2} + O_{\ve}(P^{n-\frac{r}{2}+\ve}).\end{split}
\end{equation}

Moving on to estimating $E(P)$, we get from Lemma~\ref{lemmahard} that 
\begin{align*}
E(P) &\ll  \Delta^{-\frac{1}{2}}\|F\|^{-\frac{n}{2}+\frac{3r}{4}-\frac{1}{2}}P^{n-\frac{r}{2}-1}\times \\
&\quad\sum_{\substack{x \ll Q \\ x \text{ dyadic}}}x^{-n+\frac{r}{2}-1}\sum_{\substack{\cc \in \ZZ^n \\ 0 < |\cc| \ll C}}|\cc|^{1-\frac{r}{2}}\sum_{x\leq q \leq 2x}|S_q(\cc)|.
\end{align*}
Set $\dd = S^t\cc$. Let $E_1(P)$ denote the contribution to $E(P)$ from non-zero vectors $\cc$ such that $(d_{r+1},\ldots,d_{n}) = \mathbf{0}$ and let $E_2(P)$ denote the contribution from non-zero vectors $\cc$ for which there exists $r+1 \leq j \leq n$ such that $d_j \neq 0$. 

We begin by analysing the sum $E_1(P)$. We further subdivide $E_1(P)$ as follows. Put
\begin{align*}
E_{1,1}(P) &= \Delta^{-\frac{1}{2}}\|F\|^{-\frac{n}{2}+\frac{3r}{4}-\frac{1}{2}}P^{n-\frac{r}{2}-1}\sum_{\substack{x \ll Q \\ x \text{ dyadic}}}x^{-n+\frac{r}{2}-1}\times \\
&\quad \sum_{\substack{0 < |\cc| \ll C \\ (d_{r+1},\ldots,d_n) = \mathbf{0} \\ Q^*(\cc) \neq 0}}|\cc|^{1-\frac{r}{2}}\sum_{x\leq q \leq 2x}|S_q(\cc)|
\end{align*}
and
\begin{align*}
E_{1,2}(P) &= \Delta^{-\frac{1}{2}}\|F\|^{-\frac{n}{2}+\frac{3r}{4}-\frac{1}{2}}P^{n-\frac{r}{2}-1}\sum_{\substack{x \ll Q \\ x \text{ dyadic}}}x^{-n+\frac{r}{2}-1}\times \\
&\quad\sum_{\substack{0 < |\cc| \ll C \\ (d_{r+1},\ldots,d_n) = \mathbf{0} \\ Q^*(\cc) = 0}}|\cc|^{1-\frac{r}{2}}\sum_{x\leq q \leq 2x}|S_q(\cc)|.
\end{align*}
Then $E_1(P) \leq E_{1,1}(P) + E_{1,2}(P)$.
By Lemmas~\ref{lemmaavg1} and~\ref{lemmaavg2} we get
\begin{align*}
E_{1,1}(P) &\ll_{\ve}
\|F\|^{-\frac{n}{2}+\frac{3r}{4}-\frac{1}{2}+\ve}P^{n-\frac{r}{2}-1+\ve}\times \\
&\quad \left(\Delta^{-\frac{\kappa}{2r}}P^{\frac{\kappa}{2}} + \Delta^{-\frac{1}{2r}}P^{\frac{1}{2}}\right)
\sum_{\substack{0 < |\cc| \ll C \\ (d_{r+1},\ldots,d_n) = \mathbf{0} \\ Q^*(\cc) \neq 0}}|\cc|^{1-\frac{r}{2}}.  \\
\end{align*}
As $\|F\| \leq P$, we see that
\begin{align*}
E_{1,1}(P) &\ll_{\ve} 
\Delta^{-\frac{1}{2r}}\|F\|^{-\frac{n}{2}+\frac{3r}{4}-\frac{1}{2}}P^{n-\frac{r}{2}-\frac{1}{2}}(\|F\|P)^{\ve}
\sum_{\substack{0 < |\cc| \ll C \\ (d_{r+1},\ldots,d_n) = \mathbf{0}}}|\cc|^{1-\frac{r}{2}},
\end{align*}
where we have omitted the condition $Q^*(\cc) \neq 0$ in the last expression. As a result, we get
\begin{equation}\label{eq:e11p}
E_{1,1}(P) \ll_{\ve} \Delta^{-\frac{1}{2r}}\|F\|^{\frac{r}{2}+\ve}P^{n-\frac{r}{2}-\frac{1}{2}+\ve}.
\end{equation}

Similarly, using Lemmas~\ref{lemmaavg1} and~\ref{lemmaavg2} and the fact that $\|F\| \leq P$, we get that
\begin{align*}
E_{1,2}(P) &\ll_{\ve} 
\Delta^{-\frac{1}{r}+\frac{\kappa}{2r}}\|F\|^{-\frac{n}{2}+\frac{3r}{4}-\frac{1}{2}+\ve}P^{n-\frac{r}{2}-\frac{\kappa}{2}+\ve}
\sum_{\substack{0 < |\cc| \ll C \\ (d_{r+1},\ldots,d_n) = \mathbf{0} \\ Q^*(\cc) = 0}}|\cc|^{1-\frac{r}{2}}.
\end{align*}
By~\cite[Theorem 2]{HB02}, the number of integer solutions $|\cc| \ll C$ to $Q^*(\cc) = 0$ is $O_{\ve}(C^{n-2+\ve})$, whence we get that
\begin{equation}\label{eq:e12p}
E_{1,2}(P) \ll_{\ve} \Delta^{-\frac{1}{r}+\frac{\kappa}{2r}}\|F\|^{\frac{r}{2}-1+\ve}P^{n-\frac{r}{2}-\frac{\kappa}{2}+\ve}.
\end{equation}
Finally, turning to $E_2(P)$, we get
\begin{align*}
E_2(P) &\ll_{\ve} \Delta^{-\frac{1}{2}}\|F\|^{-\frac{n}{2}+\frac{3r}{4}-\frac{1}{2}}P^{n-\frac{r}{2}-1}\sum_{\substack{x \ll P \\ x \text{ dyadic}}}x^{-n+\frac{r}{2}-1}\times \\
&\quad \sum_{\substack{0 < |\cc| \ll C \\ d_j \neq 0 \text{ for some}\\ r+1 \leq j \leq n}}|\cc|^{1-\frac{r}{2}}\sum_{x\leq q \leq 2x}|S_q(\cc)|.
\end{align*}
Then by Lemma~\ref{lemmaavg3} we get that
\begin{equation*}
\begin{split}
E_2(P) &\ll_{\ve} \|F\|^{\frac{r}{2}+\ve}P^{n-\frac{r}{2}-1+\ve}.
\end{split}
\end{equation*}
Note that $E_2(P) \ll E_{1,1}(P)$ as $\|F\| \leq P$. It is also clear that the error terms in~\eqref{eq:mp} are smaller than $E_{1,1}(P)$ and $E_{1,2}(P)$. This completes the proof of Theorem~\ref{propbh19} when $(m_{r+1},\ldots,m_n) = \mathbf{0}$. 

\subsubsection{Case 2: when there exists $r+1\leq j \leq n$ such that $m_j \neq 0$}
We proceed exactly as in the previous case. The only difference is that we will now use Lemma~\ref{lemmaavg4} to estimate the exponential sums. 

We begin by calculating the main term. Let $J=\frac{\kappa+r}{2}-2$. Then the error term in~\eqref{eq:mpe0} can be replaced by $O(\mathcal{D}^{\frac{1}{2}}\|F\|^{\ve})$. As a result, the expression for the main term in~\eqref{eq:mp} remains unchanged.

For the error term $E(P)$ we get
\begin{align*}
E(P) &\ll  \Delta^{-\frac{1}{2}}\|F\|^{-\frac{n}{2}+\frac{3r}{4}-\frac{1}{2}}P^{n-\frac{r}{2}-1}\times \\
&\quad \sum_{\substack{x \ll Q \\ x \text{ dyadic}}}x^{-n+\frac{r}{2}-1}\sum_{\substack{\cc \in \ZZ^n \\ 0 < |\cc| \ll C}}|\cc|^{1-\frac{r}{2}}\sum_{x\leq q \leq 2x}|S_q(\cc)| \\
&\ll_{\ve}  \|F\|^{\frac{3r}{4}-\frac{n+1}{2}+\ve}P^{n-\frac{r}{2}-1+\ve}\sum_{\substack{\cc \in \ZZ^n \\ 0 < |\cc| \ll C}}|\cc|^{1-\frac{r}{2}}\\
&\ll_{\ve} \|F\|^{\frac{r-1}{2}+\ve}P^{n-\frac{r}{2}-1+\ve}.
\end{align*}
Note that the error term above is smaller than the error terms $E_{1,1}(P)$ and $E_{2,2}(P)$ in~\eqref{eq:e11p} and~\eqref{eq:e12p}. This completes the proof of  Theorem~\ref{propbh19}.

\bibliographystyle{amsplain}
\bibliography{qbib}

\providecommand{\bysame}{\leavevmode\hbox to3em{\hrulefill}\thinspace}
\providecommand{\MR}{\relax\ifhmode\unskip\space\fi MR }
\providecommand{\MRhref}[2]{%
  \href{http://www.ams.org/mathscinet-getitem?mr=#1}{#2}
}
\providecommand{\href}[2]{#2}
\begin{thebibliography}{10}

\bibitem{BD08}
T.~D. Browning and R.~Dietmann, \emph{On the representation of integers by
  quadratic forms}, Proc. Lond. Math. Soc. (3) \textbf{96} (2008), no.~2,
  389--416. \MR{2396125}

\bibitem{BH19}
T.~D. Browning and L.~Q. Hu, \emph{Counting rational points on biquadratic
  hypersurfaces}, Adv. Math. \textbf{349} (2019), 920--940. \MR{3944106}

\bibitem{D03}
Rainer Dietmann, \emph{Small solutions of quadratic {D}iophantine equations},
  Proc. London Math. Soc. (3) \textbf{86} (2003), no.~3, 545--582. \MR{1974390}

\bibitem{DFI93}
W.~Duke, J.~Friedlander, and H.~Iwaniec, \emph{Bounds for automorphic
  {$L$}-functions}, Invent. Math. \textbf{112} (1993), no.~1, 1--8.
  \MR{1207474}

\bibitem{HB96}
D.~R. Heath-Brown, \emph{A new form of the circle method, and its application
  to quadratic forms}, J. reine angew. Math. \textbf{481} (1996), 149--206.
  \MR{1421949}

\bibitem{HB02}
\bysame, \emph{The density of rational points on curves and surfaces}, Ann. of
  Math. (2) \textbf{155} (2002), no.~2, 553--595. \MR{1906595}

\bibitem{HBP17}
D.~R. Heath-Brown and L.~B. Pierce, \emph{Simultaneous integer values of pairs
  of quadratic forms}, J. reine angew. Math. \textbf{727} (2017), 85--143.
  \MR{3652248}

\bibitem{KR24}
V.~Vinay Kumaraswamy and Nick Rome, \emph{On a conjecture of {W}ooley and
  counting rational points on cubic hypersurfaces}, Preprint (2024), 66 pp.

\bibitem{S51}
Carl~Ludwig Siegel, \emph{Indefinite quadratische {F}ormen und
  {F}unktionentheorie. {I}}, Math. Ann. \textbf{124} (1951), 17--54. \MR{67930}

\bibitem{W61}
G.~L. Watson, \emph{Indefinite quadratic {D}iophantine equations}, Mathematika
  \textbf{8} (1961), 32--38. \MR{130212}

\end{thebibliography}
\end{document}